\documentclass[12 pt]{amsart}

\usepackage{hyperref}
\usepackage{etex}
\usepackage[shortlabels]{enumitem}
\usepackage{amsmath}
\usepackage{amsxtra}
\usepackage{amscd}
\usepackage{amsthm}
\usepackage{amsfonts}
\usepackage{amssymb}
\usepackage{eucal}
\usepackage[all]{xy}
\usepackage{graphicx}
\usepackage{tikz-cd}
\usepackage{mathrsfs}
\usepackage{subfiles}
\usepackage{mathpazo}
\usepackage[colorinlistoftodos, textsize=tiny]{todonotes}
\usepackage{morefloats}
\usepackage{pdfpages}
\usepackage{thm-restate}
\usepackage[utf8]{inputenc}
\usepackage{epigraph}
\usepackage{csquotes}
\usepackage[margin=1in]{geometry}

\graphicspath{ {images/} }

\RequirePackage{color}
\definecolor{myred}{rgb}{0.75,0,0}
\definecolor{mygreen}{rgb}{0,0.5,0}
\definecolor{myblue}{rgb}{0,0,0.65}

\usepackage{hyperref}
\hypersetup{citecolor=blue}
\usepackage{tikz}
\usetikzlibrary{matrix,arrows,decorations.pathmorphing}

\theoremstyle{plain}
\newtheorem{theorem}{Theorem}[section]

\newtheorem{lemma}[theorem]{Lemma}
\newtheorem{corollary}[theorem]{Corollary}
\theoremstyle{definition}
\newtheorem{definition}[theorem]{Definition}
\newtheorem{remark}[theorem]{Remark}

\newtheorem{question}[theorem]{Question}

\theoremstyle{remark}

\numberwithin{equation}{section}
  
\newcommand\nc{\newcommand}
\nc\on{\operatorname}
\nc\renc{\renewcommand}

\newcommand\bc{\mathbb C}

\DeclareMathOperator\xsym{Sym_2}

\newcommand\scc{\mathscr C}

\newcommand\scj{\mathscr J}

\newcommand\sco{\mathscr O}

\newcommand \ra{\rightarrow}
\newcommand \xra{\xrightarrow}

\DeclareMathOperator\spec{\text{Spec}}

\newcommand*{\shom}{\mathscr{H}\kern -.5pt om}
\newcommand*{\stor}{\mathscr{T}\kern -.5pt or}
\newcommand*{\sext}{\mathscr{E}\kern -.5pt xt}
\newcommand \mg{{\mathscr M_g}}
\newcommand \ag{{\mathscr A_g}}

\makeatletter
\newcommand{\customlabel}[2]{\protected@write \@auxout {}{\string \newlabel {#1}{{#2}{\thepage}{#2}{#1}{}} }\hypertarget{#1}{#2}}

\DeclareMathOperator\id{id}

\renewcommand\hom{\mathrm{Hom}}

\DeclareMathOperator\pic{Pic}

\DeclareMathOperator\spn{Span}

\DeclareMathOperator\sym{Sym}

\DeclareFontFamily{U}{wncy}{}
\DeclareFontShape{U}{wncy}{m}{n}{<->wncyr10}{}
\DeclareSymbolFont{mcy}{U}{wncy}{m}{n}
\DeclareMathSymbol{\Sha}{\mathord}{mcy}{"58}

\def\listtodoname{List of Todos}
\def\listoftodos{\@starttoc{tdo}\listtodoname}

\title{The infinitesimal Torelli problem}
\author{Aaron Landesman}

\usepackage{microtype}
\begin{document}

\maketitle
\begin{abstract}
	Let $g \geq 2$ and let the Torelli map denote the map sending a genus $g$ curve to its principally polarized Jacobian.
	We verify the well known fact that the map induced on tangent spaces by the Torelli map is dual to the multiplication map
	$\mathrm{Sym}^2 H^0(C, \omega_C) \rightarrow H^0(C, \omega_C^{\otimes 2})$.
\end{abstract}

\section{Introduction}

Let $\mg$ denote the moduli stack of
curves of genus $g$, and let $\ag$ denote the moduli stack of principally polarized abelian varieties of genus $g$.
Throughout, we assume $g \geq 2$ and work over a fixed algebraically closed field $k$.
Let $\mg \xra{\tau_g} \ag$ denote the Torelli map sending a curve to its principally polarized Jacobian.
Let $[C] \in \mg$ and let $J$ denote the principally polarized Jacobian of $C$.
The main question we would like to address is the following: 
\begin{question}
	\label{question:tangent-map}
	Why is $T_{[C]}\tau_g : T_{[C]}\mg \ra T_{[J]} \ag$,
	the map induced on tangent spaces by the Torelli map,
	dual to the natural multiplication map 
	$\sym^2 H^0(C, \omega_C) \ra H^0(C, \omega_C^{\otimes 2})$?
\end{question}

This is answered in \autoref{theorem:torelli-dual}.
The content of this note is well known to the experts.
The answer to this question has essentially already appeared in \cite[Theorem 2.6]{oortS:local-torelli-problem}.
However, some details which were labeled as obvious 
in the proof of \cite[Theorem 2.6]{oortS:local-torelli-problem} took us some time to verify.
Specifically, \autoref{lemma:commute-tangent-space} and \autoref{theorem:torelli-dual} are stated without justification in \cite[Theorem 2.6]{oortS:local-torelli-problem}.
For this reason, we decided to write this explanation.

An analytic proof (over $\bc$) is also given in \cite[Lemma 3]{griffiths:some-remarks-and-examples-on-continuous-systems-and-moduli}.
We originally learned this statement from \cite{mathoverlow:is-the-torelli-map-an-immersion}.

\subsection{Notation and Overview}
Let $C$ be a smooth proper connected curve over an algebraically closed field $k$ and let $J := \pic^0_{C/k}$
denote its Jacobian.
For any $x \in C(k)$, we have an Abel-Jacobi map
\begin{align*}
	j_x: C & \rightarrow J \\
	y & \mapsto \sco_C(y-x),
\end{align*}
which is a closed immersion.
We therefore obtain an induced map $T_C \ra j_x^* T_J$, where $T_X$ the tangent sheaf of $X$.
Hence, we obtain a map 
\begin{align}
	\label{eq:tangent-map}
	dj_x: H^1(C, T_C) \ra H^1(C, j_x^* T_J).
\end{align}

The plan for the remainder of this note is as follows:
First, we relate $dj_x$ to the map induced on tangent spaces by the Torelli map $T_{[C]} \mg \ra T_{[J]} \ag$ in \autoref{lemma:commute-tangent-space}.
Then, in \autoref{section:ag-tangent}, we recall relevant identifications of the Torelli map on tangent spaces to
essentially answer \autoref{question:tangent-map}
modulo a description of $dj_x$ as dual to a certain multiplication map of sections.
This description is proven in \autoref{theorem:torelli-dual}.

\subsection{Acknowledgements}
Many thanks to Bogdan Zavyalov for listening carefully to the arguments in this note.

\section{Identifying the map on tangent space with $dj_x$}

Note that there are natural Kodaira Spencer maps $T_{[C]} \mg \ra H^1(C, T_C)$ and $T_{[J]} \ag \ra H^1(J, T_J)$ which sends a deformation to its corresponding cohomology class. 
(The Kodaira Spencer map is explicitly described in terms of a certain boundary map in cohomology, see for example \cite[p. 168]{oortS:local-torelli-problem}.)
From these, we obtain a diagram
\begin{equation}
	\label{equation:ks-tangent-space}
	\begin{tikzcd} 
		T_{[C]} \mg \ar {rr} \ar {d} && T_{[J]} \ag \ar {d} \\
		H^1(C, T_C) \ar {rd}{dj_x} && H^1(J, T_J) \ar{ld} \\
		\qquad & H^1(C, j_x^* T_J).
\end{tikzcd}\end{equation}
It is stated in \cite[p. 169]{oortS:local-torelli-problem} that \eqref{equation:ks-tangent-space} commutes.
\begin{lemma}
	\label{lemma:commute-tangent-space}
	The diagram \eqref{equation:ks-tangent-space} commutes.
\end{lemma}
\begin{proof}
	Start with a deformation $\scc$ over $D = \spec k[\varepsilon]/(\varepsilon^2)$ corresponding to an element $[\scc] \in T_{[C]} \mg$.
	Let $\scj$ denote $\pic^0_{\scc/D}$, the corresponding deformation of $J$ under the Torelli map.
	Then, $\scj$ corresponds to a class $[\scj] \in H^1(J, T_J)$ via the Kodaira Spencer map.
	Now, by smoothness of $C$, $x$ lifts to a point $\widetilde{x} \in \scc(D)$. 
	Then, the Abel-Jacobi map $j_x: C \ra J$ lifts to a map $j_{\widetilde{x}}: \scc \ra \scj$ which is again a closed immersion. 
	
	We would like to show that $[\scc]$ and $[\scj]$ map to the same element of $H^1(C, j_x^* T_J)$.
We now verify this by working with explicit cocycles.
To this end, choose a cover $U_i$ for $J$ and let $V_i := j_x^{-1}(U_i)$, so that we obtain a resulting cover $V_i$ for $C$.
Let $U_{ij} := U_i \cap U_j$ and $V_{ij} := V_i \cap V_j$.
We can lift $[\scj]$ to a cocycle $u_{ij} \in H^0(U_{ij}, T_J|_{U_{ij}})$.
Similarly, $[\scc]$ corresponds to a cocycle $v_{ij} \in H^0(V_{ij}, T_C|_{V_{ij}})$.
Under the correspondence between sections of the tangent sheaf, derivations, and automorphisms of $U_{ij} \times D$ restricting to the identity on the closed fiber,
$u_{ij}$ corresponds to an automorphism $(1 + u_{ij}) =: \phi_{ij}: U_{ij} \times D \ra U_{ij} \times D$ and $v_{ij}$ corresponds to the an automorphism
$(1+v_{ij}) =: \psi_{ij}: V_{ij} \times D \ra V_{ij} \times D$
making
\begin{equation}
	\label{equation:v-to-u-deformation}
	\begin{tikzcd} 
		V_{ij} \times D \ar {r} \ar {d}{\psi_{ij}} & U_{ij} \times D \ar {d}{\phi_{ij}} \\
		V_{ij} \times D \ar {r} & U_{ij} \times D
\end{tikzcd}\end{equation}
a fiber square.

It suffices to show that the image of $u_{ij}$ under $f: H^0(U_{ij}, T_J|_{U_{ij}}) \ra H^0(V_{ij}, j_x^* T_J|_{V_{ij}})$ agrees with the image of
$v_{ij}$ under the map $h: H^0(V_{ij}, T_C|_{V_{ij}})\ra H^0(V_{ij}, T_J|_{V_{ij}})$.
We can restrict both these maps to any given point $y \in C$.
Letting $h_y: T_y(C) \ra T_{j_x(y)} (J)$ denote the restriction of $h$ to $y$, it suffices to check $h_y(v_{ij}|_y) = u_{ij}|_{j_x(y)}$, where both are considered as elements of
the $g$-dimensional vector space
$T_{j_x(y)}(J)$.
This is now a concrete calculation over the an infinitesimal neighborhood of the single point $y$. 

Applying the functor $\hom_D(D, \bullet)$ to \eqref{equation:v-to-u-deformation} and viewing the outputs as $k[\varepsilon]/(\varepsilon^2)$ modules,
we obtain the following commutative diagram
\begin{equation}
	\label{equation:v-to-u-at-y}
	\begin{tikzcd} 
		k[\varepsilon]/(\varepsilon^2) \ar {r}{h_y'}\ar{d}{\psi_{ij}|_y} & k[\varepsilon]/(\varepsilon^2) \otimes H^0(j_x^* T_J) \ar{d}{\phi_{ij}|_y} \\
		k[\varepsilon]/(\varepsilon^2) \ar {r}{h_y'}  & k[\varepsilon]/(\varepsilon^2) \otimes H^0(j_x^* T_J),
\end{tikzcd}\end{equation}
where $h_y'$ is identified with $h_y$ by $h_y'(1+\varepsilon v) = 1 + \varepsilon h_y(v)$. 
This identification follows from the correspondence between derivations and automorphisms over the dual numbers
restricting to the identity on the special fiber.

Let $\omega_1, \ldots, \omega_g$ be a basis for $T_J|_e$.
By functoriality of the pairing between the tangent space and cotangent space applied to $j_x: C \ra J$,
we see that for $v \in T_y(C)$ and $\omega \in T{j_x(y)}(J)^\vee$, we have $\langle v, j^*(\omega) \rangle_{T_y(C)} = \langle dj_x(v), \omega \rangle_{T_{j_x(y)}(J)}$.
In what follows, we use $\omega|_y(v)$ as shorthand for $\langle v, j^*(\omega) \rangle_{T_y(C)}$.
Via the compatibility of tangent space pairings mentioned above, we find
the map $h_y'$ is given by sending 
\begin{align*}
1 + \varepsilon v \mapsto \left( 1 + \varepsilon \omega_1|_y(v), \ldots, 1 + \varepsilon \omega_g|_y(v) \right).
\end{align*}
Therefore, if the map $\psi_{ij}|_y$ is given by $\times (1 + \varepsilon a)$, we see that commutativity of \eqref{equation:v-to-u-at-y} forces $\phi_{ij}|_y$ to be given by
$\left( \times (1 + \varepsilon \omega_1|_y(a) ), \ldots, \times (1 +\varepsilon \omega_g|_y(a))\right)$.
This tells us that $u_{ij}|_{j_x(y)}$ is given by $\left( \omega_1|_y(v_{ij}|_y), \ldots, \omega_g|_y(v_{ij}|_y) \right) = h_y(v_{ij}|_y)$,
as we wanted to show.
\end{proof}

\section{Identifying the tangent space to $\ag$}
\label{section:ag-tangent}

In this subsection, we explain some relevant background from \cite{oortS:local-torelli-problem} in order to identify the tangent spaces of $\ag$ and $\mg$.
Before continuing, let us make explicit our notational conventions for symmetric powers of vector spaces.
\begin{definition}
	\label{definition:}
	For $V$ a vector space, define $\xsym V$ as the kernel of $V \ra \wedge^2 V$, where $\wedge^2 V = V/ \spn(v\otimes v : v \in V)$.
Note that $\xsym V \simeq \sym^2 V$ in characteristic $p \neq 2$, but differs in characteristic $2$. Here $\sym^2 V$ denotes the natural quotient of $V \otimes V$ by the span of $v \otimes w - w \otimes v$ for $v, w \in V$.
\end{definition}

In this subsection, we explain how to identify the map $T_{[C]} \mg \ra T_{[J]} \ag$ with a map $H^1(C, T_C) \ra \xsym H^1(J, \sco_J)$ which is dual to a map
\begin{align*}
\sym^2 H^0(C, \omega_C) \ra H^0(C, \omega_C^{\otimes 2}).
\end{align*}
This map turns out to be the multiplication map, and in order to check this, 
(essentially using 
\cite[Theorem 2.6]{oortS:local-torelli-problem})
we explain why it suffices to show $dj_x$ is dual to the multiplication map
$H^0(C, \omega_C)^{\otimes 2} \otimes H^0(C, \omega_C) \ra H^0(C, \omega_C^{\otimes 2})$.

Having justified commutativity of \eqref{equation:ks-tangent-space},
recall that the tangent space to $\mg$ at $C$ is identified with $H^1(C, T_C)$.
Recall that the principal polarization induces an isomorphism $\widehat{J} \simeq J$, where $\widehat{J}$ denotes the dual abelian variety, and hence
an isomorphism on tangent spaces $H^1(J, \sco_J) \simeq H^0(J, T_J)$.
Using this isomorphism, the tangent space
to $\ag$ at $J$ is identified with 
\begin{align*}
\xsym H^1(J, \sco_J) \subset H^1(J, \sco_J) \otimes H^1(J, \sco_J) \simeq H^1(J, \sco_J) \otimes H^0(J, T_J) \simeq H^1(J, T_J).
\end{align*}
This identification is implicit in the proof of \cite[Theorem 2.6]{oortS:local-torelli-problem}, but also see \cite[Theorem 3.3.11(iii)]{sernesi:deformations-of-algebraic-schemes}
for much of the relevant deformation theory.

Under the above identifications, 
we saw $T_{[J]} \ag \ra H^1(J, T_J)$ is identified with a map
\begin{align*}
	\xsym H^1(J, \sco_J) \ra H^1(J, T_J) \simeq H^1(J, \sco_J) \otimes H^0(J,T_J).
\end{align*}
This in turn can be identified with a map
\begin{align*}
\xsym H^0(C, \omega_C)^\vee \ra H^0(C, \omega_C)^\vee \otimes H^0(C, \omega_C)^\vee. 
\end{align*}
These identifications are given by 
\begin{align*}
H^1(J, \sco_J) \simeq H^1(C, \sco_C) \simeq H^0(C, \omega_C)^\vee
\end{align*}
via
pullback and Serre duality, and 
\begin{align*}
H^0(J, T_J) \simeq H^0(J, \Omega^1_J)^\vee \simeq H^0(C, \Omega^1_C)^\vee \simeq H^0(C, \omega_C)^\vee
\end{align*}
via the tangent space pairing, pullback and the isomorphism
$\Omega^1_C \simeq \omega_C$.
\begin{corollary}
	\label{corollary:sym-to-tensor}
Under the above identifications, in order to identify the map on tangent spaces $T_{[C]} \mg \ra T_C \ag$ as dual to the natural multiplication map
on differentials $\sym^2 H^0(C, \omega_C) \ra H^0(C, \omega_C^{\otimes 2})$,
it suffices to identify the map
$dj_x$ as dual to the multiplication map
$H^0(C, \omega_C) \otimes H^0(C, \omega_C) \ra H^0(C, \omega_C^{\otimes 2})$.
\end{corollary}
\begin{proof}
	Under the above identifications, 
	it is shown in \cite[Theorem 2.6]{oortS:local-torelli-problem} that
the map 
$T_{[J]} \ag \ra H^1(J, T_J) \simeq H^1(C, j_x^* T_J)$
is dual to the natural projection map
$H^0(C, \omega_C) \otimes H^0(C, \omega_C) \ra \sym^2 H^0(C, \omega_C)$.
Hence, using 
\autoref{lemma:commute-tangent-space},
in order to identify the map on tangent spaces $T_{[C]} \mg \ra T_C \ag$ as dual to the natural map
$\sym^2 H^0(C, \omega_C) \ra H^0(C, \omega_C^{\otimes 2})$
it suffices to identify the map
$dj_x$ as dual to the map
$H^0(C, \omega_C) \otimes H^0(C, \omega_C) \ra H^0(C, \omega_C^{\otimes 2})$,
as we will do below in \autoref{theorem:torelli-dual}.
\end{proof}

\section{Identifying the dual of $dj_x$}
The goal for the remainder of this note is to identify the Serre dual of $dj_x$ as a natural multiplication map.
By Serre duality, we have a natural identification
\begin{align*}
	H^1(C, T_C) \simeq H^0(C, \Omega^1_C \otimes \omega_C)^\vee.
\end{align*}
\begin{remark}
	\label{remark:}
	Note in the above, we are using $\Omega^1_C$ to denote the sheaf of differentials and $\omega_C$
to denote the dualizing sheaf. Of course, they are isomorphic, but we believe it will help clarify things
later to call them by different names.
\end{remark}

Further, we have a duality inducing an isomorphism 
\begin{align*}
H^1(C, j_x^* T_J) \simeq \left( H^0(C, \Omega^1_C) \otimes H^0(C, \omega_C)\right)^\vee
\end{align*}
given by the pairing
\begin{equation}
\begin{aligned}
	&H^1(C, j_x^* T_J) \otimes H^0(C, \Omega^1_C) \otimes H^0(C, \omega_C) \\
	&\xra{f_1}
	H^1(C, \sco_C \otimes H^1(C, \sco_C) ) \otimes H^0(C, \Omega^1_C) \otimes H^0(C, \omega_C) \\
	&\xra{f_2}
	H^1(C, \sco_C) \otimes H^1(C, \sco_C) \otimes H^0(C, \Omega^1_C) \otimes H^0(C, \omega_C) \\
	&\xra{f_3}
	H^1(C, \sco_C) \otimes H^1(C, \omega_C) \\
	&\xra{f_4}
	k
\end{aligned}
	\label{eq:jacobian:duality}
\end{equation}
where $f_3$ is given by $a \otimes b \otimes c \otimes d \mapsto b(c) \cdot a \otimes d$, 
via the natural pairing between the tangent space to $J$, identified with $H^1(C, \sco_C)$, and
the cotangent space to $J$, identified with $H^0(C, \Omega^1_C) \simeq H^0(J, \Omega^1_J)$.
Finally, the map $f_4$ is given by Serre duality.
\begin{remark}
	\label{remark:}
	We note that the composition $f_4 \circ f_3 \circ f_2 \circ f_1$ is identified with the Serre duality pairing
	$H^1(C, j_x^*T_J) \otimes H^0(C, \omega_C \otimes j_x^* \Omega^1_J) \ra k$ under the identification
	\begin{align*}
	H^0(C, \omega_C \otimes j_x^* \Omega^1_J) &\simeq H^0(C, \omega_C \otimes H^0(J, \Omega^1_J)) \\
	&\simeq H^0(C, \omega_C \otimes H^0(C, \Omega^1_C)) \\
	&\simeq H^0(C, \omega_C )\otimes H^0(C, \Omega^1_C).
	\end{align*}
\end{remark}

The main result of this note is the following:
\begin{theorem}
	\label{theorem:torelli-dual}
	Using the notation and identifications above, the map $dj_x$ of \eqref{eq:tangent-map} is dual to the multiplication map $\mu: H^0(C, \Omega^1_C) \otimes H^0(C, \omega_C) \ra H^0(C, \Omega^1_C \otimes \omega_C)$.
	Further,
	under the identification of \eqref{equation:ks-tangent-space}, $dj_x$ determines the the map $T_{[C]}\mg \ra T_{[J]}\ag$
	which can be identified as dual to the multiplication map $\sym^2 H^0(C, \omega_C) \ra H^0(C, \omega_C^{\otimes 2})$, via the identifications of \autoref{section:ag-tangent}.
\end{theorem}
\begin{proof}
By \autoref{corollary:sym-to-tensor},
	using \autoref{lemma:commute-tangent-space} and \cite[Theorem 2.6]{oortS:local-torelli-problem},
	it suffices to identify $dj_x$ as dual to $\mu$.
	As mentioned before, \cite[Theorem 2.6]{oortS:local-torelli-problem}
	does state that $dj_x$ is dual to $\mu$, though in the proof it is stated that this is ``obvious from the preceding arguments'' and we did not understand why at first, so
	we explain this for the remainder of the proof.

	Under the pairings given above, we need to verify that for $v \in H^1(T_C), \alpha \in H^0(C, \Omega^1_C),$ and $\beta \in H^0(C, \omega_C)$, we have
	\begin{align*}
		(\alpha \otimes \beta) (dj_x(v)) = \left( \mu(\alpha \otimes \beta) \right)(v)
	\end{align*}
Equivalently, we wish to show
\begin{equation}
	\label{equation:}
	\begin{tikzcd} 
		H^1(C, T_C) \otimes H^0(C, \Omega^1_C) \otimes H^0(C, \omega_C) \ar {r}{\id \otimes \mu} \ar {d}{dj_x \otimes \id \otimes \id} & H^1(C, T_C) \otimes H^0(C, \Omega^1_C \otimes \omega_C) \ar {d} \\
		H^1(C, j_x^* T_J) \otimes H^0(C, \Omega^1_C) \otimes H^0(C, \omega_C) \ar {r} & k
\end{tikzcd}\end{equation}
commutes,
where the bottom map is the composition of 
\eqref{eq:jacobian:duality}
and the right vertical map is the Serre duality pairing.

Note that the above diagram fits into a larger diagram
\begin{equation}
	\label{equation:decomposed-square}
\begin{tikzpicture}[baseline= (a).base]
\node[scale=.80] (a) at (0,0){
	\begin{tikzcd}[column sep = tiny]
		H^1(C, T_C) \otimes H^0(C, \Omega^1_C) \otimes H^0(C, \omega_C) \ar {rr}{\id \otimes \mu} \ar {dd}{dj_x \otimes \id \otimes \id} \ar{rd}{\nu \otimes \id} && H^1(C, T_C) \otimes H^0(C, \Omega^1_C \otimes \omega_C) \ar {dd} \\
		& H^1(C, \sco_C) \otimes H^0(C, \omega_C) \ar{rd}{f_4}& \\
		H^1(C, j_x^* T_J) \otimes H^0(C, \Omega^1_C) \otimes H^0(C, \omega_C) \ar {rr}{f_4 \circ f_3 \circ f_2 \circ f_1} \ar{ur}{f_3 \circ f_2 \circ f_1} & & k
\end{tikzcd}
};
\end{tikzpicture}
\end{equation}
for $\nu: H^1(C, T_C) \otimes H^0(C, \Omega^1_C) \ra H^1(C, T_C \otimes \Omega^1_C) \simeq H^1(C, \sco_C)$
the natural multiplication map
and $f_i$ the maps from \eqref{eq:jacobian:duality}.
We want to show this diagram commutes.

Observe that the upper right hand triangle of \eqref{equation:decomposed-square} commutes
by associativity of the tensor product and cup product: The two maps 
$H^1(C, T_C) \otimes H^0(C, \Omega^1_C) \otimes H^0(C, \omega_C) \ra k$ via passing in different ways around the upper right hand triangle
are given by collapsing the three terms
of the source in different orders, and then applying the trace map from Serre duality.
Further, the bottom triangle of \eqref{equation:decomposed-square} commutes by construction.

Therefore, it suffices to show commutativity of the left hand triangle of \eqref{equation:decomposed-square}.
Since, all maps in the left hand triangle of \eqref{equation:decomposed-square} are the identity on $H^0(C, \omega_C)$, it is in turn equivalent to show commutativity
of
\begin{equation}
	\label{equation:pullback-and-cup-cohomology}
	\begin{tikzcd} 
		H^1(C, T_C) \otimes H^0(C, \Omega^1_C) \ar {r}{dj_x \otimes \id} \ar[bend right]{dddr}{\nu} & H^1(C, j_x^* T_J) \otimes H^0(C, \Omega^1_C) \ar {d}{f_1} \\
	& H^1(C, \sco_C \otimes H^1(C, \sco_C)) \otimes H^0(C, \Omega^1_C)  \ar{d}{f_2}\\
	& H^1(C, \sco_C) \otimes H^1(C, \sco_C) \otimes H^0(C, \Omega^1_C)  \ar{d}{f_3}\\
	& H^1(C, \sco_C). 
\end{tikzcd}\end{equation}
To show commutativity of \eqref{equation:pullback-and-cup-cohomology},
we note that it is in fact the induced diagram obtained by taking the first cohomology
of the following diagram of sheaves on $C$:
\begin{equation}
	\label{equation:pullback-and-cup-cohomology-sheaf}
	\begin{tikzcd} 
		T_C \otimes_{\sco_C} H^0(C, \Omega^1_C) \ar {r}{\delta j_x \otimes \id} \ar{dd}{n_1} & j_x^* T_J \otimes_{\sco_C} H^0(C, \Omega^1_C) \ar {d}{m_1} \\
		& \sco_C \otimes H^1(C, \sco_C) \otimes H^0(C, \Omega^1_C)  \ar{d}{m_2}\\
		T_C \otimes \Omega^1_C \ar{r}{n_2} &  \sco_C. 
\end{tikzcd}\end{equation}
Since $C$ is reduced, to check these two maps of sheaves
$T_C \otimes_{\sco_C} H^0(C, \Omega^1_C) \ra \sco_C$ agree, we can check they agree over each point $y \in C$.

We now finish the proof by checking commutativity of \eqref{equation:pullback-and-cup-cohomology-sheaf} over $y \in C$.
Let $v \in T_y(C) := T_C|_y$ and let $\omega \in H^1(J, T_J)$.
Under the map $\delta j_x : T_C \ra j_x^* T_J$, an element $v \in T_y(C)$ is sent to $\delta j_x(v)$, which can be viewed as the functional $\omega \mapsto \langle dj_x(v), \omega|_{j_x(y)} \rangle$.
Therefore, $(m_2 \circ m_1 \circ (\delta j_x \otimes \id)) (v \otimes j_x^*(\omega)) = \langle \delta j_x(v), \omega|_{j_x(y)} \rangle$.
On the other hand, $n_2 \circ n_1 (v \otimes j_x^*(\omega)) = \langle v, j_x^*(\omega)|_y \rangle$.
Commutativity of the diagram then follows because
\begin{align*}
\langle v, j_x^*(\omega)|_y \rangle_{T_y(C)} =  \langle \delta j_x(v),\omega|_{j_x(y)} \rangle_{T_{j_x(y)}(J)}
\end{align*}
as the pairing between the tangent and cotangent spaces is functorial along the map $C \ra J$.
\end{proof}

\bibliographystyle{alpha}
\bibliography{/home/aaron/Dropbox/master}

\def\cprime{$'$} \providecommand{\noopsort}[1]{}
\begin{thebibliography}{{VA.}}

\bibitem[Gri67]{griffiths:some-remarks-and-examples-on-continuous-systems-and-moduli}
Phillip~A. Griffiths.
\newblock Some remarks and examples on continuous systems and moduli.
\newblock {\em J. Math. Mech.}, 16:789--802, 1967.

\bibitem[OS79]{oortS:local-torelli-problem}
Frans Oort and Joseph Steenbrink.
\newblock The local {{Torelli}} problem for algebraic curves.
\newblock {\em Journ{\'e}es de G{\'e}ometrie Alg{\'e}brique d'Angers},
  1979:157--204, 1979.

\bibitem[Ser06]{sernesi:deformations-of-algebraic-schemes}
Edoardo Sernesi.
\newblock {\em Deformations of algebraic schemes}, volume 334 of {\em
  Grundlehren der Mathematischen Wissenschaften [Fundamental Principles of
  Mathematical Sciences]}.
\newblock Springer-Verlag, Berlin, 2006.

\bibitem[{VA.}]{mathoverlow:is-the-torelli-map-an-immersion}
{VA. (https://mathoverflow.net/users/1784/va)}.
\newblock Is the torelli map an immersion?
\newblock MathOverflow.
\newblock URL:https://mathoverflow.net/q/9338 (version: 2009-12-19).

\end{thebibliography}

\end{document}